\documentclass[11pt]{amsart}
\usepackage{geometry}
\usepackage{amsmath}
\usepackage{amsfonts}
\usepackage{amsthm}
\usepackage{amssymb}
\allowdisplaybreaks[2]
\newtheorem{theorem}{Theorem}[section]

\theoremstyle{definition}

\newtheorem{lemma}{Lemma}[section]
\newtheorem{corollary}{Corollary}[section]

\theoremstyle{theorem}
\newtheorem{other}{\bf Theorem}              

\DeclareMathOperator{\defeq}{\overset{def}=}

\numberwithin{equation}{section}

\newcommand{\hol}{{\mathcal Hol}}
\DeclareMathOperator{\og}{O}

\newcommand{\h}{\mathcal{H}}
\newcommand{\hu}{\mathcal{H}_\mu}
\newcommand{\Hu}{\mathcal{H}_\mu}
\newcommand{\iu}{I_\mu}
\newcommand{\om}{\omega}
\newcommand{\lao}{\Lambda(p,\omega)}

\newcommand{\aw}{\mathcal{AW}}
\newcommand{\eiteta}{e^{i\theta}}

\newcommand{\B}{\mathcal{B}}

\def\D{{\mathbb D}}

\begin{document}

\title{Mean Lipschitz spaces and a generalized Hilbert operator}

\author{Noel Merch\'{a}n}
\address{An\'{a}lisis Matem\'{a}tico, Facultad de Ciencias, Universidad de M\'{a}laga, 29071 M\'{a}laga, Spain}
\email{noel@uma.es}

\thanks{This research is supported in part by a grant from \lq\lq El Ministerio de
Econom\'{\i}a y Competitividad\rq\rq , Spain (MTM2014-52865-P) and
by a grant from la Junta de Andaluc\'{\i}a FQM-210. The author is also supported by a grant from \lq\lq El Ministerio de de
Educaci\'{o}n, Cultura y Deporte\rq\rq , Spain (FPU2013/01478).}



\begin{abstract} If $\mu $ is a positive Borel measure on the interval $[0, 1)$ we
let $\mathcal H_\mu $ be the Hankel matrix $\mathcal H_\mu =(\mu
_{n, k})_{n,k\ge 0}$ with entries $\mu _{n, k}=\mu _{n+k}$, where,
for $n\,=\,0, 1, 2, \dots $, $\mu_n$ denotes the moment of order $n$
of $\mu $. This matrix induces formally the operator
$$\mathcal{H}_\mu (f)(z)=
\sum_{n=0}^{\infty}\left(\sum_{k=0}^{\infty}
\mu_{n,k}{a_k}\right)z^n$$ on the space of all analytic functions
$f(z)=\sum_{k=0}^\infty a_kz^k$, in the unit disc $\mathbb{D} $. This is a
natural generalization of the classical Hilbert operator. In this
paper we study the action of the operators $\mathcal H_\mu $ on mean Lipschitz
spaces of analytic functions.
\end{abstract}

\maketitle

\section{Introduction and main results}\label{intro}
Let $\mathbb D $ be the unit disc in the complex plane $\mathbb
{C}$, and let $\hol (\mathbb D )$ denote the space of all analytic
functions in $\mathbb D$. For $0<r<1$ and $f\in \hol (\mathbb D )$,
we set
$$M_ p(r, f)=\left (
\frac{1}{2\pi }\int_ {-\pi }\sp \pi \left \vert f(re\sp {i\theta })
\right \vert \sp pd\theta \right )\sp {1/p}, \quad 0<p<\infty , $$
\smallskip
$$M_ \infty (r, f)=\max_ {\vert z\vert =r}\vert f(z)\vert .$$
For $0<p\leq \infty $ the Hardy space $H\sp p$ consists of those
functions $f$, analytic in $\mathbb D$, for which
$$\Vert f \Vert  _ {H\sp p}=\sup_ {0<r<1}
M_ p(r, g)<\infty .$$ We refer to \cite{D} for the theory of Hardy
spaces.
\par
The space $BMOA$ consists of those functions $f\in H\sp 1$ whose
boundary values have bounded mean oscillation on $\partial \mathbb
D$. The Bloch space $\B$ consists of all analytic functions $f$ in
$\D$ with bounded invariant derivative:
$$f\in \mathcal B \,\,\,\Leftrightarrow\,\,\,\Vert f\Vert _
{\mathcal B }\defeq \vert f(0)\vert +\sup\sb {z \in {\mathbb D
}}\,(1-|z|\sp2)\,|f\sp\prime(z)|<\infty \,. $$ We mention \cite{ACP,
G:BMOA, Zhu-book} as excellent references for these spaces. Let us
recall that  $BMOA \subsetneq \B$.
\par
If\, $\mu $ is a finite positive Borel measure on $[0, 1)$ and $n\,
= 0, 1, 2, \dots $, we let $\mu_n$ denote the moment of order $n$ of
$\mu $, that is, $\mu _n=\int _{[0,1)}t^n\,d\mu (t),$ and we let
$\mathcal H_\mu $  be the Hankel matrix $(\mu _{n,k})_{n,k\ge 0}$
with entries $\mu _{n,k}=\mu_{n+k}$. The matrix $\mathcal H_\mu $
induces formally an operator, also denoted $\mathcal H_\mu $, on
spaces of analytic functions in the following way: if\,
 $f(z)=\sum_{k=0}^\infty a_kz^k\in \hol (\D )$
we define
\begin{equation*}\label{H}
\mathcal{H}_\mu (f)(z)= \sum_{n=0}^{\infty}\left(\sum_{k=0}^{\infty}
\mu_{n,k}{a_k}\right)z^n,
\end{equation*}
whenever the right hand side makes sense and defines an analytic
function in $\D $.
\par\smallskip If $\mu $ is the Lebesgue measure on $[0,1)$ the matrix
$\mathcal H_\mu $ reduces to the classical Hilbert matrix \,
$\mathcal H= \left ({(n+k+1)^{-1}}\right )_{n,k\ge 0}$, which
induces the classical Hilbert operator $\h$. The Hilbert operator is
known to be well defined on $H^1$ and bounded from $H^p$ into
itself, if $1<p<\infty $, but not if $p=1$ or $p=\infty $
\cite{DiS}.
\par
The question of describing the measures $\mu $ for which the
operator $\mathcal H_\mu $ is well defined and bounded on distinct
spaces of analytic functions has been studied in a good number of
papers (see \cite{Bao-Wu, Ch-Gi-Pe, Ga-Pe2010, GM1, GM2, Pell, Pow,
Wi}). The measures in question are Carleson-type measures.
\par If  $I\subset \partial\D$ is an
interval, $\vert I\vert $ will denote the length of $I$. The
\emph{Carleson square} $S(I)$ is defined as
$S(I)=\{re^{it}:\,e^{it}\in I,\quad 1-\frac{|I|}{2\pi }\le r <1\}$.
\par If $\, s>0$ and $\mu$ is a positive Borel  measure on  $\D$,
we shall say that $\mu $
 is an $s$-Carleson measure
  if there exists a positive constant $C$ such that
\[
\mu\left(S(I)\right )\le C{|I|^s}, \quad\hbox{for any interval
$I\subset\partial\D $}.
\]
\par  A $1$-Carleson
measure will be simply called a Carleson measure.
\par If $\mu$ is a positive Borel measure on $\D$, $0\le
\alpha <\infty $, and $0<s<\infty $ we say that $\mu$ is an
 $\alpha$-logarithmic $s$-Carleson measure \cite{Zhao} if there exists a positive
 constant $C$ such that
 \[\frac{
\mu\left(S(I)\right )\left(\log \frac{2\pi }{\vert I\vert }\right
)^\alpha }{|I|^s}\le C, \quad\hbox{for any interval
$I\subset\partial\D $}.
\]
\par
A positive Borel measure $\mu $ on $[0, 1)$ can be seen as a Borel
measure on $\mathbb D$ by identifying it with the measure $\tilde
\mu $ defined by $$ \tilde \mu (A)\,=\,\mu \left (A\cap [0,1)\right
),\quad \text{for any Borel subset $A$ of $\mathbb D$}.$$  In this
way a positive Borel measure $\mu $ on $[0, 1)$ is an $s$-Carleson
measure if and only if there exists a positive constant $C$ such
that
\[
\mu\left([t,1)\right )\le C(1-t)^s, \quad 0\le t<1.
\]
We have a similar statement for $\alpha$-logarithmic $s$-Carleson
measures.
\par Widom \cite[Theorem\,\@3.\,\@1]{Wi} (see also
\cite[Theorem\,\@3]{Pow} and \cite[p.\,\@42,
Theorem\,\@7.\,\@2]{Pell}) proved that $\mathcal H_\mu $ is a
bounded operator from $H^2$ into itself if and only $\mu $ is a
Carleson measure. Galanopoulos and Pel\'{a}ez \cite{Ga-Pe2010}
studied the operators $\mathcal H_\mu $ acting on $H^1$. The action
of $\mathcal H_\mu $ on the Hardy spaces $H^p$, $0<p\le \infty $,
has been studied in \cite{Ch-Gi-Pe, GM1, GM2}. The papers \cite{GM1}
and \cite{GM2} study also the operators $H_\mu $ acting on distinct
subspaces of the Bloch space, including $BMOA$, Besov spaces, and
the $Q_s$-spaces.
\par In this paper we shall study the operators $H_\mu $ acting on
mean Lipschitz spaces of analytic functions.
\par
If $f\in \hol (\mathbb D )$ has a non-tangential limit $f(e\sp
{i\theta })$ at almost every $e\sp {i\theta }\in \partial\mathbb D$
and $\delta >0$,  we define
\begin{align*}\omega _ p(\delta , f)= &\sup_{0<\vert t\vert \leq \delta }\left
(\frac{1}{2\pi }\int _{-\pi }\sp \pi \left \vert f(e\sp {i(\theta
+t)})-f(e\sp {i\theta })\right \vert \sp p\,
d\theta \right )\sp {1/p}, \quad\text{if $1\leq p<\infty $},\\
\omega _ \infty (\delta , f)=& \sup_ {0<\vert t\vert \leq \delta
}\left (\operatornamewithlimits{ess.sup}_ {\theta\in [-\pi, \pi ]}
\vert f(e\sp {i(\theta +t)})-f(e\sp {i\theta })\vert\right
).\end{align*}  Then $\omega _ p(\cdot , f)$ is the integral modulus
of continuity of order $p$ of the boundary values $f(e\sp {i\theta
})$ of $f$.
\par
Given $1\leq p\leq \infty $ and $0<\alpha \leq 1$, the mean
Lipschitz space $\Lambda _ \alpha \sp p$ consists of those functions
$f\in\hol (\mathbb D)$ having a non-tangential limit almost
everywhere for which $\omega _ p(\delta , f)=O(\delta \sp \alpha )$,
as $\delta \to 0$. If $p=\infty $ we write $\Lambda _ \alpha $
instead of $\Lambda _ \alpha \sp \infty $. This is the usual
Lipschitz space of order $\alpha $.
\par
A classical result of Hardy and Littlewood \cite{HL-1932} (see also
\cite[Chapter~5]{D}) asserts that for $1\leq p\leq \infty $ and
$0<\alpha \leq 1$, we have that $\Lambda _ \alpha \sp p\subset H\sp
p$ and \begin{equation}\label{lamba-p-alpha}\Lambda _ \alpha \sp
p=\left \{ f\in \hol (\mathbb D) : M_ p(r, f\sp\prime )=\og\left
(\frac{1}{(1-r)\sp {1-\alpha }}\right ) \right \}.\end{equation} It
is known that if $1<p<\infty$ and $\alpha>\frac{1}{p}$ then each
$f\in\Lambda^p_\alpha$ is bounded and has a continuous extension to
the closed unit disc (\cite{BSS}, p.88). This is not true for
$\alpha=\frac{1}{p}$, because the function $f(z)=\log(1-z)$ belongs
to $\Lambda^p_{1/p}$ for all $p\in(1,\infty)$. By a theorem of Hardy
and Littlewood \cite[Theorem 5.9]{D} and  of \cite[Theorem 2.5]{BSS}
we have
$$\Lambda^p_{1/p}\subset\Lambda^q_{1/q}\subset BMOA\quad 1\le p<q<\infty .$$
\par
The inclusion $\Lambda^p_{1/p}\subset BMOA$, $1\le p<\infty$ was
proved to be sharp in a very strong sense in \cite{BGM, G:Normal,
G:Illi}
 using
the following generalization of the spaces $\Lambda _ \alpha \sp p$
which occurs frequently in the literature. Let $\omega :[0, \pi
]\rightarrow [0, \infty )$ be a continuous and increasing function
with $\omega (0)=0$ and $\omega (t)>0$ if $t>0$. Then, for $1\leq
p\leq \infty $, the mean Lipschitz space $\Lambda (p, \omega )$
consists of those functions $f\in H\sp p$ such that
$$\omega_ p(\delta , f)=O(\omega (\delta )),\quad\hbox{as $\delta \to 0$}.$$
With this notation we have $\Lambda \sp p_ \alpha =\Lambda (p,
\delta\sp\alpha )$.
\par
The question of finding conditions on $\omega $ so that it is
possible to obtain results on the spaces $\Lambda (p, \omega )$
analogous to those proved by Hardy and Littlewood for the spaces
$\Lambda \sp p_ \alpha $ has been studied by several authors (see
\cite{BS1, BS2, BSS}). We shall say that $\omega $ satisfies the
Dini condition or that $\omega $ is a Dini-weight if there exists a
positive constant $C$ such that
$$\int_ 0\sp\delta \frac{\omega (t)}{t}\, dt\leq C\omega (\delta ),
\quad 0<\delta <1.$$ We shall say that $\omega $ satisfies the
condition $b_1$ or that $\omega \in b_1$ if there exists a positive
constant $C$ such that
$$\int_ \delta\sp\pi
\frac{\omega (t)}{t\sp {2}}\, dt\leq C\frac{\omega (\delta )}{\delta
}, \quad 0<\delta <1.$$
\par In order to simplify our notation, let $\mathcal A\mathcal W$ denote the family of
all functions $\omega :[0, \pi ]\rightarrow [0, \infty )$ which
satisfy the following conditions: \begin{itemize}\item[(i)] $\omega
$ is continuous and increasing in $[0, \pi ]$.
\item[(ii)] $\omega (0)=0$ and $\omega (t)>0$ if $t>0$.
\item[(iii)] $\omega $ is a Dini-weight.
\item[(iv)] $\omega $ satisfies
the condition $b_1$.\end{itemize}
\par The elements of $\mathcal A\mathcal W$ will be called admissible weights.
Characterizations and examples of admissible weights can be found in
\cite{BS1, BS2}. \par Blasco and de Souza extended the above
mentioned result of Hardy and Littlewood showing in \cite[Th.
2.1]{BS1} that if $\omega\in \mathcal A\mathcal W$ then,
$$\Lambda (p, \omega )=\left \{ \text{$f$ analytic in $\D $ :
$M_ p(r, f\sp\prime )=\og\left ( \frac{\omega (1-r)}{1-r}\right ),$
as $r\to 1$}\right \} .$$
\par
In \cite{BGM, G:Normal,G:Illi} it is proved that if $1\le p<\infty $
and $\om$ is an admissible weight such that
 $$\frac{\om(\delta)}{\delta^{1/p}}\to\infty ,\text{ as }\delta\to 0,$$
 then there exists a function $f\in\lao$ which is a not a normal
 function (see \cite{ACP} for the definition). Since any Bloch function is normal, if follows that for such admissible
 weights $\om $ one has that $\Lambda (p, \om )\not\subset \mathcal
 B$.

\par One of the main results in \cite{GM2} is the following one.
\begin{other}[\cite{GM2}]\label{th-X} Let
$\mu $ be a positive Borel measure on $[0,1)$ and let $X$ be a
Banach space of analytic functions in $\mathbb D$ with
$\Lambda^2_{1/2}\subset X\subset \mathcal B$. Then the following
conditions are equivalent.
\begin{itemize}
\item[(i)] The operator $\mathcal H_\mu $ is well defined in $X$ and, furthermore, it is a bounded operator from $X$ into he Bloch space $\mathcal B$.
\item[(ii)] The operator $\mathcal H_\mu $ is well defined in $X$ and, furthermore, it is a bounded operator from $X$ into $\Lambda^2_{1/2}$.
\item[(iii)] The measure $\mu $ is a
$1$-logarithmic $1$-Carleson measure.
\item[(iv)] $\int_{[0,1)}t^n\log\frac{1}{1-t}d\mu (t)\,=\,\og \left
(\frac{1}{n}\right )$.
\end{itemize}
\end{other}
\par\medskip
A key ingredient in the proof of Theorem\,\@\ref{th-X} is the fact
that for any space $X$ with $\Lambda^2_{1/2}\subset X\subset
\mathcal B$ the functions $f\in X$ of the form
$f(z)=\sum_{n=0}^\infty a_nz^n$ whose sequence of Taylor
coefficients $\{ a_n\} $ is a decreasing sequence of non-negative
numbers are the same. Indeed, for such a function $f$ and such a
space $X$ we have that $f\in X\,\,\Leftrightarrow\,\, a_n=\og \left
(\frac{1}{n}\right ).$ This result remains true if we substitute
$\Lambda^2_{1/2}$ by $\Lambda^p_{1/p}$ for any $p>1$. That is, the
following result holds:
\begin{lemma}\label{ThDeclambdap} Suppose that $1<p<\infty$ and let
$f\in \hol (\D )$ be of the form  $f(z)=\sum_{n=0}^\infty a_nz^n$
with $\{ a_n\}_{n=0}^\infty $ being a decreasing sequence of
nonnegative numbers. If $X$ is a subspace of $\hol (\D )$ with
$\Lambda^p_{1/p}\subset X\subset \mathcal B$, then
$$f\in X\quad \Leftrightarrow\quad a_n=\og \left (\frac{1}{n}\right
).$$
\end{lemma}

Lemma\,\@\ref{ThDeclambdap} is a consequence of the following one
which will be proved in Section\,\@\ref{Lips}.

\begin{lemma}\label{ThDec}
Let $1<p<\infty$, $\om\in\aw$ and let $f(z)=\sum_{n=0}^\infty
a_nz^n$ with $\{ a_n\}_{n=0}^\infty $ being a decreasing sequence of
nonnegative numbers. Then \begin{equation}\label{dec-coef-omega}f\in
\lao\quad \Leftrightarrow\quad a_n\,=\,\og \left
(\frac{\om(1/n)}{n^{1-1/p}}\right ).\end{equation}
\end{lemma}

Using Lemma\,\@\ref{ThDeclambdap} and following the proof of
Theorem\,\@\ref{th-X} in \cite{GM2}, we obtain

\begin{theorem}\label{ext-p-may-1} Suppose that  $1<p<\infty$. Let $\mu $ be a positive Borel measure on $[0,1)$ and let $X$ be a
Banach space of analytic functions in $\mathbb D$ with
$\Lambda^p_{1/p}\subset X\subset \mathcal B$. Then the following
conditions are equivalent.
\begin{itemize}
\item[(i)] The operator $\mathcal H_\mu $ is well defined in $X$ and, furthermore, it is a bounded operator from $X$ into he Bloch space $\mathcal B$.
\item[(ii)] The operator $\mathcal H_\mu $ is well defined in $X$ and, furthermore, it is a bounded operator from $X$ into $\Lambda^p_{1/p}$.
\item[(iii)] The measure $\mu $ is a
$1$-logarithmic $1$-Carleson measure.
\item[(iv)] $\int_{[0,1)}t^n\log\frac{1}{1-t}d\mu (t)\,=\,\og \left
(\frac{1}{n}\right )$.
\end{itemize}
\end{theorem}
\par As an immediate consequence of Theorem\,\@\ref{ext-p-may-1} we
obtain the following result.
\begin{corollary} Let $\mu $ be a positive Borel measure on $[0,1)$
and $1<p<\infty $. Then the operator $\mathcal H_\mu $ is well
defined in $\Lambda^p_{1/p}$ and, furthermore, it is a bounded
operator from $\Lambda^p_{1/p}$ into itself if and only if $\mu $ is
a $1$-logarithmic $1$-Carleson measure.
\end{corollary}

\par\medskip Let us turn our attention now to the spaces $\lao $
with $\frac{\om(\delta)}{\delta^{1/p}}\nearrow\infty$,
$\delta\searrow 0$ which, as noted before, are not included in the
Bloch space. We have the following result which shows that the
situation is different from the one covered in
Theorem\,\@\ref{ext-p-may-1}.
\begin{theorem}\label{thHlao}
Let $1<p<\infty$, $\om\in\aw$ with $\frac{\om(\delta)}{\delta^{1/p}}\nearrow \infty$ when $\delta\searrow 0$. The following conditions are equivalent:
\begin{itemize}
\item[(i)] The operator $\hu$ is well defined in $\lao$ and, furthermore, it is a bounded operator from $\lao$ into itself.
\item[(ii)] The measure $\mu$ is a Carleson measure.
\end{itemize}
\end{theorem}

\par The proofs of Lemma\,\@\ref{ThDec}
and Theorem\,\@\ref{thHlao} will be presented in
Section\,\@\ref{Lips}.
 We close this section
noticing that, as usual, we shall be using the convention that
$C=C(p, \alpha ,q,\beta , \dots )$ will denote a positive constant
which depends only upon the displayed parameters $p, \alpha , q,
\beta \dots $ (which sometimes will be omitted) but not  necessarily
the same at different occurrences. Moreover, for two real-valued
functions $E_1, E_2$ we write $E_1\lesssim E_2$, or $E_1\gtrsim
E_2$, if there exists a positive constant $C$ independent of the
arguments such that $E_1\leq C E_2$, respectively $E_1\ge C E_2$. If
we have $E_1\lesssim E_2$ and  $E_1\gtrsim E_2$ simultaneously then
we say that $E_1$ and $E_2$ are equivalent and we write $E_1\asymp
E_2$.

\section{Proofs of the main results}\label{Lips}
We start recalling that for a function $f(z)=\sum_{n=0}^\infty
a_nz^n$ analytic in $\D,$ the polynomials $\Delta_jf$ are defined as
follows:
\[
\Delta_jf(z)=\sum_{k=2^j}^{2^{j+1}-1} a_kz^k, \quad\text{for $j\ge
1$},
\]
$$\Delta_0f(z)=a_0+a_1z.$$ The proof of Lemma\,\@\ref{ThDec} is based
in the following result of Girela and Gonz\'{a}lez
\cite[Theorem\,\@2]{GG}.

\begin{other}[]\label{ThGG}
Let $1<p<\infty$ and let $\om$ be an admissible weight. If
$f\in\hol(\D)$ with $f(z)=\sum_{n=0}^\infty a_nz^n$ then
$$f\in\lao \,\,\Leftrightarrow \,\,\|\Delta_N f\|_{H^p}=O\left(\om \left(\frac{1}{2^N}\right)\right).$$
\end{other}

\begin{proof}[Proof of Lemma \ref{ThDec}]
By Lemma A of \cite{Pav-dec}, since $a_n\searrow 0$, we have
$$\|\Delta_N f\|_{H^p}\asymp a_{2^N} 2^{N(1-1/p)},\quad N\ge 1.$$
So by Theorem \ref{ThGG} we have that
$$ f\in\lao\Leftrightarrow a_{2^N} \lesssim \frac{\om\left(1/2^N\right)}{2^{N(1-1/p)}},\quad N\ge 1. $$
This easily implies (\ref{dec-coef-omega}).
\end{proof}

\begin{lemma}\label{LemMom} Suppose that $1<p<\infty $.
Let $\nu$ be a positive Borel measure on $[0,1)$, and let
$\om\in\aw$ satisfying that $x^{-1/p}\om(x)\nearrow\infty$, as
$x\searrow 0$. Then following conditions are equivalent:
\begin{itemize}\item[(i)]
$\nu_n\lesssim \frac{\om(1/n)}{n^{{1-1/p}}},\, n\ge2$.
\item[(ii)] $\nu([b,1])\lesssim (1-b)^{1-1/p}
\om(1-b),\, b\in[0,1).$
\end{itemize}
\end{lemma}
\begin{proof}
Suppose (i). Then we have that
\begin{align*}
1\gtrsim& \frac{n^{1-1/p}\,\nu_n}{\om(1/n)}=\frac{n^{1-1/p}}{\om(1/n)}\int_{[0,1)} t^n\,d\nu(t)\ge \frac{n^{1-1/p}}{\om(1/n)}\int_{[1-1/n,1)} t^n\,d\nu(t)
\\ \ge& \frac{n^{1-1/p}}{\om(1/n)}\nu([1-1/n,1))\,\left(1-\frac{1}{n}\right)^n
\\ \ge& \frac{n^{1-1/p}}{\om(1/n)}\nu([1-1/n,1))\,\inf\limits_{m\ge 2}\left(1-\frac{1}{m}\right)^m
\\ \gtrsim& \frac{n^{1-1/p}}{\om(1/n)}\nu([1-1/n,1)).
\end{align*}
So $\nu([1-1/n,1))\lesssim \frac{\om(1/n)}{n^{1-1/p}}$ for $n\ge2$.
\par Let now $b\in[1/2,1)$. There exists $n\ge 2$ such that $1-\frac{1}{n}\le b<1-\frac{1}{n+1}$ so using the above we have that
$$\nu([b,1))\le \nu([1-1/n,1))\lesssim \frac{\om(1/n)}{n^{1-1/p}}.$$
This, and the facts that $\om(1/n)n^{1/p}\le\om(1/(n+1))(n+1)^{1/p}$
and that the weight $\om$ increases  give (ii).
\par Suppose now (ii). Then
\begin{align*}
\nu_n&=\int_{[0,1)} t^n\,d\nu(t)=n\int_0^1 \nu([t,1))t^{n-1}\,dt
\\ &\lesssim n\int_0^1 (1-t)^{1-1/p} \om(1-t)t^{n-1}\,dt
\\ &= n \int_0^{1-\frac{1}{n}}\,+\,\int_{1-\frac{1}{n}}^1\left((1-t)^{1-1/p} \om(1-t)t^{n-1}\,dt\right).
\end{align*}
The first integral can be estimated bearing in mind that
$(1-t)^{-1/p}\om(1-t)\nearrow\infty$ when $t\nearrow 1$ as follows
\begin{align*}
&n \int_0^{1-\frac{1}{n}} (1-t)^{1-1/p} \om(1-t)t^{n-1}\,dt
\\ \le& n^{1+1/p}\om(1/n)\int_0^{1-\frac{1}{n}} (1-t)t^{n-1}\,dt
\\ =& n^{1+1/p}\om(1/n)\left(1-\frac{1}{n}\right)^n\left(\frac{1}{n} -\frac{n-1}{n(n+1)}\right)
\\ \lesssim&\frac{\om(1/n)}{n^{1-1/p}}.
\end{align*}
To estimate of the second integral we use that
$(1-t)^{1-1/p}\om(1-t)\searrow0$ when $t\nearrow 1$ to obtain
\begin{align*}
&n \int_{1-\frac{1}{n}}^1 (1-t)^{1-1/p} \om(1-t)t^{n-1}\,dt
\\ \le &n^{1/p}\om(1/n)\int_{1-\frac{1}{n}}^1 t^{n-1}\,dt
\\ =&\frac{\om(1/n)}{n^{1-1/p}}\left(1-\left(1-\frac{1}{n}\right)^n\right)
\\ \lesssim &\frac{\om(1/n)}{n^{1-1/p}}.
\end{align*}
Then (i) follows.
\end{proof}

\begin{proof}[Proof of Theorem \ref{thHlao}]
(i) $\Rightarrow$ (ii) Suppose that $\hu:\lao\to\lao$ is bounded. By
Lemma\,\@\ref{ThDec}  we have that the function $f$ defined by
$f(z)=\sum_{n=1}^\infty \frac{\om(1/n)}{n^{1-1/p}}z^n$ belongs to
the space $\lao$ so, by the hypothesis, $\Hu (f)$ belongs also to
$\lao$. Now
$$\Hu(f)(z)=\sum_{n=0}^\infty\left(\sum_{k=1}^\infty \frac{\om(1/k)}{k^{1-1/p}}\,\mu_{n+k} \right)z^n.$$
Notice that $\sum_{k=1}^\infty \frac{\om(1/k)}{k^{1-1/p}}\,\mu_{n+k}\searrow 0,\, n\to\infty$, so using again Lemma \ref{ThDec} it holds that
$$\sum_{k=1}^\infty \frac{\om(1/k)}{k^{1-1/p}}\,\mu_{n+k}=\int_{[0,1)} t^n\sum_{k=1}^\infty \frac{\om(1/k)}{k^{1-1/p}}\,t^k\,d\mu(t)
\lesssim \frac{\om(1/n)}{n^{1-1/p}}, $$
that is, the moments of the measure $\nu$ defined by
$$d\nu(t)=\sum_{k=1}^\infty \frac{\om(1/k)}{k^{1-1/p}}\,t^k\,d\mu(t)$$
satisfy that $$\nu_n\lesssim \frac{\om(1/n)}{n^{1-1/p}}, $$
so by Lemma \ref{LemMom} we have that $\nu([b,1))\lesssim (1-b)^{1-1/p} \om(1-b)$, $b\in[0,1)$.
\par According to the definition of the measure
\begin{align*}
(1-b)^{1-1/p} \om(1-b)&\gtrsim \nu([b,1))=\int_{[b,1)}d\nu(t)
\\ &=\int_{[b,1)}\sum_{k=1}^\infty \frac{\om(1/k)}{k^{1-1/p}}\,t^k\,d\mu(t)
\\ &\ge \mu\left([b,1)\right)\sum_{k=1}^\infty \frac{\om(1/k)}{k^{1-1/p}}\,b^k
\end{align*}
and the sum can be estimated as follows
\begin{align*}
\sum_{k=1}^\infty \frac{\om(1/k)}{k^{1-1/p}}\,b^k &\asymp \int_1^\infty \frac{\om(1/x)}{x^{1-1/p}}\,b^x\,dx
\\ &\ge \int_1^{\frac{1}{1-b}} \frac{\om(1/x)}{x^{1-1/p}}\,b^x\,dx
\\ &\ge (1-b)^{1-1/p}\om(1-b)b^{\frac{1}{1-b}}\left(\frac{1}{1-b}-1\right)
\\ &\gtrsim \frac{\om(1-b)}{(1-b)^{1/p}}.
\end{align*}
Finally, putting all together we have that
$$\mu([b,1))\lesssim 1-b$$ so $\mu$ is a Carleson measure.

\par (ii) $\Rightarrow$ (i)
To prove this implication we need to use the integral operator
$I_\mu $ considered in \cite{Ch-Gi-Pe, Ga-Pe2010, GM1, GM2} which is
closely related to the operator $\mathcal H_\mu$.
\par If $\mu $ is a positive Borel measure on $[0,1)$ and $f\in \hol (\D )$,
we shall write throughout the paper
\begin{equation*}I_\mu
(f)(z)=\int_{[0,1)}\frac{f(t)}{1-tz}\,d\mu (t),\end{equation*}
whenever the right hand side makes sense and defines an analytic
function in $\D $. It turns out that the operators $H_\mu $ and
$I_\mu$ are closely related. Indeed, as shown in the just mentioned
papers, it turns out that if $f$ is good enough $H_\mu (f)$ and
$I_\mu (f)$ are well defined and coincide.
\par
Suppose that $\mu$ is a Carleson measure supported on $[0, 1)$ and
let $f\in\lao$. We claim that
\begin{equation}\label{eqIuFin}
\int_{[0,1)}\frac{|f(t)|}{|1-tz|}\,d\mu(t)<\infty .
\end{equation}

Indeed, using Lemma 3 of \cite{GG} we have that

\begin{equation}\label{eqGrowth}
f\in\lao \Rightarrow|f(z)|\lesssim \frac{\om(1-|z|)}{(1-|z|)^{1/p}},\quad z\in\D.
\end{equation}
Then we obtain
\begin{align*}
\int_{[0,1)}\frac{|f(t)|}{|1-tz|}\,d\mu(t)&\le \frac{1}{1-|z|}\int_{[0,1)}|f(t)|\,d\mu(t)
\\ &\lesssim \frac{1}{1-|z|}\int_{[0,1)} \frac{\om(1-t)}{(1-t)^{1/p}}\,d\mu(t).
\end{align*}
If we choose $r\in[0,1)$ we can split the integral in the intervals $[0,r)$ and $[r,1)$. In the first one, as $\om$ is an increasing weight we have
\begin{align*}
\int_{[0,r)} \frac{\om(1-t)}{(1-t)^{1/p}}\,d\mu(t)&\le \om(1)\int_{[0,r)} \frac{d\mu(t)}{(1-t)^{1/p}}
\\ &\le  \om(1)\int_{[0,1)} \frac{d\mu(t)}{(1-t)^{1/p}}
\\ &\lesssim 1,
\end{align*}
because $\mu$ is a Carleson measure. Using this and the condition
$\frac{\om(\delta)}{\delta^{1/p}}\nearrow\infty$, as $\delta\searrow
0$ we can estimate the other integral as follows
\begin{align*}
\int_{[r,1)} \frac{\om(1-t)}{(1-t)^{1/p}}\,d\mu(t)&\le \frac{\om(1-r)}{(1-r)^{1/p}} \int_{[r,1)} d\mu(t)
\\ &\lesssim \om(1-r)(1-r)^{1-1/p}
\\ &\lesssim 1.
\end{align*}
So we have that for $f\in\lao$ and $z\in\D$, (\ref{eqIuFin}) holds.
This implies that $I_\mu (f)$ is well defined, and, using Fubini's
theorem and standard arguments it follows easily that $\mathcal
H_\mu (f)$ is also well defined and that, furthermore, $$\mathcal
H_\mu (f)(z)=I_\mu (f)(z),\quad z\in \mathbb D.$$ Now we have,
$$\iu(f)^\prime(z)=\int_{[0,1)}\frac{tf(t)}{(1-tz)^2}\,d\mu(t),\quad z\in\D,$$
so the mean of order $p$ of $\iu(f)^\prime $ has the form
$$M_p\left(r,\iu(f)^\prime\right)=\left(\frac{1}{2\pi}\int_{-\pi}^\pi \left|\int_{[0,1)}\frac{tf(t)}{(1-tr\eiteta)^2}\,d\mu(t)\right|^p\,d\theta\right)^{1/p}.$$
Using again (\ref{eqGrowth}), the Minkowski inequality and a
classical estimation of integrals we obtain that
\begin{align*}
M_p\left(r,\iu(f)^\prime\right)&\lesssim\int_{[0,1)} |f(t)|\left(\int_{-\pi}^\pi \frac{d\theta}{|1-tr\eiteta|^{2p}}\right)^{1/p}\,d\mu(t)
\\ &\lesssim \int_{[0,1)}\frac{ |f(t)|}{(1-tr)^{2-1/p}}\,d\mu(t)
\\ &\lesssim \int_{[0,1)}\frac{ \om(1-t)}{(1-t)^{1/p}(1-tr)^{2-1/p}}\,d\mu(t).
\end{align*}
At this point we split the integrals on the sets $[0,r)$ and
$[r,1)$.
\par In the first integral we use that $x^{-1/p}\om(x)\nearrow\infty$, as $x\searrow0$, and the fact that if $\mu$ is a Carleson measure
(so that  $\mu_n=\int_{[0,1)} t^n\,d\mu(t)\lesssim \frac{1}{n}$) to
obtain
\begin{align*}
\int_{[0,r)}\frac{ \om(1-t)}{(1-t)^{1/p}(1-tr)^{2-1/p}}\,d\mu(t)&\le \frac{ \om(1-r)}{(1-r)^{1/p}}\int_{[0,r)}\frac{d\mu(t)}{(1-tr)^{2-1/p}}
\\&\le \frac{ \om(1-r)}{(1-r)^{1/p}}\int_{[0,1)}\frac{d\mu(t)}{(1-tr)^{2-1/p}}
\\ &\lesssim \frac{ \om(1-r)}{(1-r)^{1/p}}\sum\limits_{n=1}^\infty n^{1-1/p}r^n\int_{[0,1)} t^n\,d\mu(t)
\\ &\lesssim \frac{ \om(1-r)}{(1-r)^{1/p}}\sum\limits_{n=1}^\infty\frac{r^n}{n^{1/p}}
\\ &\lesssim \frac{ \om(1-r)}{(1-r)} .
\end{align*}
In the second integral we use that $\om$ is an increasing weight and
the fact that the measure $\mu$ being  a Carleson measure is
equivalent to saying that the measure $\nu$ defined by
$d\nu(t)=\frac{d\mu(t)}{(1-t)^{1/p}}$ is a $1-\frac{1}{p}$-Carleson
measure so that the moments $\nu_n$ of $\nu$ satisfy $\nu_n\lesssim
\frac{1}{n^{1-\frac{1}{p}}}$. Then we obtain
\begin{align*}
\int_{[r,1)}\frac{ \om(1-t)}{(1-t)^{1/p}(1-tr)^{2-1/p}}\,d\mu(t)&\le  \om(1-r)\int_{[r,1)}\frac{d\nu(t)}{(1-tr)^{2-1/p}}
\\ &\le  \om(1-r)\int_{[0,1)}\frac{d\nu(t)}{(1-tr)^{2-1/p}}
\\ &\lesssim  \om(1-r)\sum\limits_{n=1}^\infty n^{1-1/p}r^n\int_{[0,1)}t^n\,d\nu(t)
\\ &\lesssim  \om(1-r)\sum\limits_{n=1}^\infty r^n
\\ &=\frac{ \om(1-r)}{(1-r)} .
\end{align*}
Therefore $\iu(f)\in\lao$ and then the operator $I_\mu $ (and hence
the operator $\mathcal H_\mu $) is bounded from $\lao $ into itself.
\end{proof}

\par
\end{document}